\documentclass[11pt]{amsart}
\usepackage[utf8]{inputenc}
\usepackage{amssymb,amsfonts}
\usepackage{euscript,mathrsfs}
\usepackage{latexsym}
\usepackage{amscd}
\usepackage{amsmath}
\usepackage{a4}
\usepackage{epsf}
\usepackage{color}
\usepackage{graphicx}
\usepackage{txfonts}
\usepackage{amsthm,amsopn}
\usepackage{verbatim}
\usepackage{microtype}
\usepackage{bbm}
\usepackage[arrow,matrix]{xy}

\newcommand{\Bbf}{\mathbf{B}}

\newcommand{\Gcal}{\mathcal{G}}
\newcommand{\Pcal}{\mathcal{P}}
\newcommand{\Rcal}{\mathcal{R}}
\newcommand{\Scal}{\mathcal{S}}
\newcommand{\Xcal}{\mathcal{X}}
\newcommand{\Ycal}{\mathcal{Y}}
\newcommand{\Zcal}{\mathcal{Z}}

\newcommand{\CC}{\mathbb{C}}
\newcommand{\KK}{\mathbb{R}}

\newcommand{\ini}{\ensuremath{\operatorname{ini}_{<}}}

\renewcommand{\subset}{\subseteq}  

\newcommand{\Rb}{\mathbb{R}}

\DeclareMathOperator*{\supp}{supp}  

\theoremstyle{plain}
\newtheorem{lemma}{Lemma}
\newtheorem{thm}[lemma]{Theorem}
\newtheorem{prop}[lemma]{Proposition}

\newtheorem{cor}[lemma]{Corollary}

\theoremstyle{definition}
\newtheorem{defi}[lemma]{Definition}

\newtheorem{ex}[lemma]{Example}

\theoremstyle{remark}
\newtheorem*{rem}{Remark}

\newcommand{\CI}[2]{\left.#1 \Perp #2 \CIhelper}
\newcommand{\CIhelper}[1][\relax]{\def\argone{#1}\def\rrrrrrrelax{\relax}
  \ifx\argone\rrrrrrrelax\right.\else\,\middle|\,#1\right.{}\fi}

\newcommand{\ol}{\overline}

\newcommand{\pin}{p_\text{\textup{in}}}
\newcommand{\xin}{x_\text{\textup{in}}}

\begin{document}

\title{Robustness and Conditional Independence Ideals}

\author{Johannes Rauh${}^{1}$, Nihat Ay${}^{1,2}$}

\email{\{rauh, nay\}@mis.mpg.de}

\address{${}^{1}$Max Planck Institute for Mathematics in the Sciences,
Inselstrasse 22, 
D-04103 Leipzig,Germany}
\address{${}^2$Santa Fe Institute,
1399 Hyde Park Road,
Santa Fe, New Mexico 87501, USA}

\date{\today}

\begin{abstract}
  We study notions of robustness of Markov kernels and probability distribution of a system that is described by $n$
  input random variables and one output random variable.  Markov kernels can be expanded in a series of potentials that
  allow to describe the system's behaviour after knockouts.  Robustness imposes structural constraints on these
  potentials.

  Robustness of probability distributions is defined via conditional independence statements.  These statements can be
  studied algebraically.  The corresponding conditional independence ideals are related to binary edge ideals.  The set
  of robust probability distributions lies on an algebraic variety.  We compute a Gr\"obner basis of this ideal and
  study the irreducible decomposition of the variety.  These algebraic results allow to parametrize the set of all
  robust probability distributions.
\end{abstract}

\maketitle

\section{Introduction}
\label{sec:introduction}

In this article we study a notion of robustness with tools from algebraic geometry.  This work has been initiated in~\cite{AyKrakauer2007:Geometric_Robustness_Theory}.  Connections to algebraic geometry have already been addressed in~\cite{HHHKR10:Binomial_Edge_Ideals}.
We consider $n$ input nodes,
denoted by $1,2,\dots,n$, and one output node, denoted by $0$.  For each $i=0,1,\dots,n$ the state of node $i$ is a
discrete random variable $X_{i}$ taking values in the finite set ${\mathcal X}_i$ of cardinality $d_{i}$.
The joint state space is the set $\tilde\Xcal=\Xcal_{0}\times\Xcal_{1}\times\dots\times\Xcal_{n}$.  For any subset
$S\subseteq\{0,\dots,n\}$ write $X_{S}$ for the random vector $(X_{i})_{i\in S}$; then $X_{S}$ is a
random variable with values in $\Xcal_{S}=\times_{i\in S}\Xcal_{i}$.  For any $x\in\tilde\Xcal$, the \emph{restriction} of
$x$ to a subset $S\subseteq\{0,\dots,n\}$ is the vector $x|_{S}\in\Xcal_{S}$ with $(x|_{S})_{i}=x_{i}$ for all $i\in S$.

We study two possible models for the computation of the output from the input: The first model is a stochastic map
(Markov kernel) $\kappa$ from ${\mathcal X}_{[n]}$ to ${\mathcal X}_0$, that is, $\kappa$ is a function
\[
    \kappa:  {\mathcal X}_{[n]} \times {\mathcal X}_0 \; \to \; [0,1], \qquad (x,y) \; \mapsto \; \kappa(x;y) \, ,
\]
satisfying $\sum_{y \in \Xcal_{0}} \kappa(x;y) = 1$ for all $x$.  The second model is a joint probability distribution $p$
of the random vector $(X_{0},X_{[n]})$.  These two models are related as follows:
The joint probability distribution $p$ of $(X_{0},X_{[n]})$ can be factorized as
\begin{equation*}
  p(y,x) = p(y|x)\pin(x),\qquad\text{ for all }(y,x)\in\tilde\Xcal,
\end{equation*}
where $\pin$ is the distribution of the input nodes and $p(y|x)$ is a conditional distribution, which need not be
unique.  Each possible choice of this conditional distribution defines a Markov kernel $\kappa(x;y) := p(y|x)$.
Conversely, when a Markov kernel $\kappa$ is given, then any input distribution $\pin(x)$ defines a joint distribution
$p(x,y) = \pin(x)\kappa(x;y)$.  The result of our analysis will not depend too much on the precise form of the input
distribution; it will turn out that only the \emph{support} $\supp(\pin):=\{x\in\Xcal : \pin(x)>0\}$ is important.
Similarly, in the analysis of the kernels, there will also be a set $\Scal$ of \emph{``relevant inputs''} that will play
an important role.

We study robustness with respect to knockouts of some of the input nodes $[n]$ in both models.  When a subset $S$ of the
input nodes is knocked out, and only the nodes in $R=[n]\setminus S$ remain, then the behaviour of the system changes.
Without further assumptions, the post-knockout function is not determined by $\kappa$ and has to be specified.  We
therefore consider a further stochastic map $\kappa_R: {\mathcal X}_{R} \times {\mathcal X}_0 \to [0,1]$ as model of the
post-knockout function. A complete specification of the function is given by the family ${(\kappa_A)}_{A \subseteq [n]}$
of all possible post-knockout functions, which we refer to as \emph{functional modalities}.  As a shorthand notation we
denote functional modalities as $(\kappa_{A})$.  The Markov kernel $\kappa$ itself, which describes the normal behaviour of the system without knockouts, can be identified with $\kappa_{[n]}$.

What does it mean for a stochastic map to be robust?  Assume that the input is in state $x$, and that we knock out a set
$S$ of inputs.  Denoting the remaining set of inputs by $R$, we say that $(\kappa_{A})$ is robust in $x = (x_R, x_S)$
against knockout of $S$, if
\begin{equation} \label{invar}
    \kappa(x_R,x_S; x_0) = \kappa_R(x_R ; x_0) \qquad \mbox{for all $x_0 \in {\mathcal X}_0$} \, .
\end{equation}
If $\mathcal{R}$ is a collection of subsets of $[n]$ and if $(\kappa_{A})$ is robust in $x$ against knockout of
$[n]\setminus R$ for all $R\in\mathcal{R}$, then we say that $(\kappa_A)$ is {\em ${\mathcal
    R}$-robust in $x$\/}.
In Section~\ref{sec:robustn-mark-kern}, we consider Gibbs representations of functional
modalities and derive structural constraints on corresponding interaction
potentials that are imposed by robustness properties. These constraints
do not depend on the configuration $x$ in which the functional
modalities are assumed to be robust.

Similar to the case of Markov kernels, the joint probability distribution $p$ does not allow to predict the behaviour of
a perturbed system.  Nevertheless, we can ask whether it is at all possible that the behaviour of the system is robust
against a given knockout of $S$.
Let $\pin$ be an input distribution, and let $(\kappa_{A})$ be the functional modalities of the system.  If $(\kappa_A)$
is 
robust against knockout of $S$ in $x$ for all $x \in \supp(\pin)$, then $X_{0}$ is stochastically independent from
$X_{S}$ given $X_{R}$ (with respect to the joint probability distribution $p(x_{0},\xin) =
\pin(\xin)\kappa(\xin;x_{0})$), where $R=[n]\setminus S$, a fact that will be denoted by $\CI{X_{0}}{X_{S}}[X_{R}]$.
In order to see this, assume $x = (x_R,x_S) \in \supp(\pin)$. Then
\begin{eqnarray*}
p(x_0 \, | \, x_R, x_S) 
   & = & \kappa(x_R , x_S ; x_0 ) \\
   & = & \kappa_R(x_R ; x_0 ) \sum_{x_{S}':(x_{R},x_{S}')\in\supp(\pin) }  
             p(x_{S}' \, | \, x_R ) \,  \\
    & = & \sum_{x_{S}':(x_{R},x_{S}')\in\supp(\pin) }
             p(x_{S}' \, | \, x_R ) \, \kappa(x_{R}, x_S' ; x_0 ) \\
   & = & \sum_{x_{S}':(x_{R},x_{S}')\in\supp(\pin) }
             p(x_{S} ' \, | \, x_R ) \, p(x_0 \, | \, x_{R}, x_S' ) \\
   & = & p(x_0 \, | \, x_R) \, .
\end{eqnarray*}
On the other hand, if 
$\CI{X_{0}}{X_{S}}[X_{R}]$ holds for a joint distribution $p$,
then any family $(\kappa_{A})$ with the property that $\kappa_{A}(x_{A};x_{0}) = p(x_{0}|x_{A})$ whenever $p(x_{A})>0$
 is robust
against knockout of $S$ for all $x\in\supp(\pin)$, where $\pin$ is the marginal input distribution.

Therefore, we call the joint probability distribution $p$ \emph{robust against knockout of $S$}, if
$\CI{X_{0}}{X_{S}}[X_{R}]$.  This means that we do not lose information about the output $X_{0}$, if the subset $S$ of
the inputs is unknown or hidden (or ``knocked out'').
Probability distributions that are robust in this sense are studied in Section~\ref{sec:robustness-and-CI}.
Section~\ref{sec:robust_functions} discusses the case that $X_{0}$ is a deterministic function of the input nodes.  The
symmetric case that $p$ is robust against knockout of any set $S$ of cardinality less than $n-k$ is studied in
Section~\ref{sec:k-robustness}.

The results about robustness are derived from an algebraic theory of generalized binomial edge ideals, which generalize
the binomial edge ideals of~\cite{HHHKR10:Binomial_Edge_Ideals} and~\cite{Ohtani09:Ideals_of_some_2-minors}.  This
theory is presented in Section~\ref{sec:gen-binom-edge-ideals}.  A Gröbner basis is constructed, and it is shown that
these ideals are radical.  Finally, a primary decomposition is computed.
Similar CI statements have recently been studied in~\cite{SwansonTaylor11:Minimial_Primes_of_CI_Ideals}.  That work
discusses what is called $(n-1)$-robustness in Section~\ref{sec:k-robustness}.

\section{Robustness of Markov kernels}
\label{sec:robustn-mark-kern}

Let ${(\kappa_A)}_{A \subseteq [n]}$ be a collection of \emph{functional modalities}, as defined in the introduction.
Instead of providing a list of all functional modes $\kappa_A$, one can describe them in 
more mechanistic terms. In order to illustrate this, we first consider an example which comes from the field of neural networks. 
In that example, we assume that the output node receives an input $x = (x_1,\dots,x_n) \in \{-1,+1\}^n$ and generates the output $+1$ with probability
\begin{eqnarray} \label{neuron}
    \kappa(x_1,\dots,x_n; +1) \; := \; \frac{1}{1 + e^{-\sum_{i = 1}^n w_i \, x_i}} \, ,
\end{eqnarray}
which implies that for an arbitrary output $x_0$ 
 \begin{eqnarray*} 
    \kappa(x_1,\dots,x_n; x_0) \; := \; \frac{ e^{\frac{1}{2} \sum_{i = 1}^n w_i \, x_i \, x_0 }}{e^{\frac{1}{2} \sum_{i = 1}^n w_i \, x_i \cdot (-1)} + 
    e^{ \frac{1}{2} \sum_{i = 1}^n w_i \, x_i \cdot ( + 1 )}} \, .
\end{eqnarray*}
This representation of the stochastic map $\kappa$ has a structure that allows inferring the function after a knockout
of a set $S$ of input nodes, by simply removing the contribution of all the nodes in $S$. In our example (\ref{neuron}),
the post-knockout function is then given as
\[
     \kappa_R(x_R; +1) \; := \; \frac{1}{1 + e^{-\sum_{i \in R} w_i \, x_i}},
\]
where $R=[n]\setminus S$.  This inference of the post-knockout function is based on the decomposition of the sum that
appears in (\ref{neuron}).  Such a decomposition is referred to as a Gibbs representation of $\kappa$ and contains more
information than $\kappa$. More generally, we consider the following model of $(\kappa_A)$
\begin{eqnarray} \label{intpot}
    \kappa_A(x_A; x_0) \; = \; \frac{e^{\sum_{B \subseteq A} \phi_B(x_B,x_0)}}{\sum_{x_0'} e^{ \sum_{B \subseteq A} \phi_B(x_B,x_0')}},
\end{eqnarray}
where the $\phi_{B}$ are functions on ${\mathcal X}_{B} \times {\mathcal X}_0$. 
Clearly, each $\kappa_A$ is strictly positive.    
Using the M\"obius inversion, it is easy to see that each strictly positive family $(\kappa_A)$ has the representation (\ref{intpot}). 
To this end, we simply set
\begin{equation}
  \label{eq:Moebiuspotentials}
  \phi_A (x_A, x_0) \; := \; \sum_{C \subseteq A} (-1)^{| A \setminus C |} \ln \kappa_C(x_C ; x_0 ) \, .
\end{equation}
Note that this representation is not unique: If an arbitrary function of $x_{A}$ is added to the function
$\phi_{A}$, then~\eqref{intpot} does not change.

A single robustness constraint has the following consequences for the $\phi_{A}$.
\begin{prop}
  \label{prop:robkernelcond}
  Let $S\subseteq[n]$ and $R=[n]\setminus S$, and let $(\kappa_A)$ be strictly positive functional modalities with
  Gibbs potentials $(\phi_{A})$.  Then $(\kappa_A)$ is robust in $x$ against knockout of $S$ if and only if
  $\sum_{\substack{B\subseteq[n], B\not\subseteq R}}\phi_{B}(x|_{B},x_{0})$ does not depend on $x_{0}$.
\end{prop}
\begin{proof}
  Denote $\tilde\phi_{A}$ the potentials defined via~\eqref{eq:Moebiuspotentials}.  Then~\eqref{invar} is equivalent to
  \begin{equation*}
    \sum_{B\subseteq[n]}\tilde\phi_{B}(x|_{B},x_{0})
    =
    \sum_{B\subseteq R}\tilde\phi_{B}(x|_{B},x_{0})
    \quad\Longleftrightarrow\quad
    \sum_{\substack{B\subseteq[n]\\B\not\subseteq R}}\tilde\phi_{B}(x_{B},x_{0})
    = 0.
  \end{equation*}
  The statement follows from the fact that $\phi_{B}(x|_{B};x_{0}) - \tilde\phi_{B}(x|_{B};x_{0})$ is independent of
  $x_{0}$ (for fixed $x$).
\end{proof}
Does ${\mathcal R}$-robustness in $x$ imply any structural constraints on $(\kappa_A)$? In order to answer this
question, we restrict attention to the case
${\mathcal R} = {\mathcal R}_k := \{ R \subseteq [n] \; : \; | R | \geq k \}$.

If $(\kappa_{A})$ is $\Rcal_{k}$-robust on a set $\Scal$, then the corresponding conditions imposed by
Proposition~\ref{prop:robkernelcond} depend on $\Scal$.  In this section, we are interested in conditions that are
independent of $\Scal$.  Such conditions allow to define sets of functional modalities that contain all
$\Rcal_{k}$-robust functional modalities for all possible sets $\Scal$.  If the set $\Scal$ (which will be the support
of the input distribution in Section~\ref{sec:robustness-and-CI}) is unknown from the beginning, then the system can
choose its policy within such a restricted set of functional modalities.

Denote $K_{k}$ the set of all functional modalities $(\kappa_{A})$ such that
there exist potentials $\phi_{A}$ 
of the form
\begin{equation*}
  \phi_{A}(x_{A};x_{0}) = \sum_{\substack{B\subseteq A \\ |B|\le k}}\Psi_{B,A}(x_{B};x_{0}),
\end{equation*}
where $\Psi_{B,A}$ is an arbitrary function $\Rb^{\Xcal_{B}\times\Xcal_{0}}\to\Rb$.
The set $K_{k}$ is called the \emph{family of $k$-interaction functional modalities}.  It contains the subset $\tilde
K_{k}$ of those functional modalities $(\kappa_{A})$ where the functions $\Psi_{B,A}$ additionally satisfy
\begin{equation*}
  (-1)^{|A|}\Psi_{B,A}(x_{B};x_{0}) = (-1)^{|A'|}\Psi_{B,A'}(x_{B};x_{0}),\qquad\text{ whenever $B\subseteq A\cap A'$ and $|B|< k$},
\end{equation*}
and
\begin{equation*}
  \sum_{l=0}^{|A'|-k} \frac{(-1)^{| A' | - l}}{ \binom{l + k}{k} }
  \Psi_{B,A}(x_{B};x_{0}) 
  =
  \sum_{l=0}^{|A|-k} \frac{(-1)^{| A | - l}}{ \binom{l + k}{k} }
  \Psi_{B,A'}(x_{B};x_{0}),\quad\text{ if $B\subseteq A\cap A'$ and $|B| = k$},
\end{equation*}
for all $x_{B}\in\Xcal_{B}$ and $x_{0}\in\Xcal_{0}$.
Both $K_{k}$ and $\tilde K_{k}$ only contain strictly positive kernels.  Therefore, we are also interested in the
resepective \emph{closures} of these two families with respect to the usual real topology on the space of matrices.

The following holds:
\begin{prop}
  \label{prop:rob-kernel-k-interactions}
  Let ${\mathcal S}$ be a subset of ${\mathcal X}_{[n]}$ and let $(\kappa_A)$ be functional modalities that
  are 
  ${\mathcal R}_k$-robust in $x$ for all $x \in {\mathcal S}$.  Then there exist functional modalities
  $(\tilde\kappa_{A})$ in the closure of $\tilde K_{k}$ such that $\kappa_{A}(x|_{A}) = \tilde\kappa_{A}(x|_{A})$
  for all $A$
  and all $x\in\Scal$.  In particular, $\tilde\kappa_{A}$ belongs to the closure of the family of $k$-interactions.
\end{prop}

\begin{proof}
Assume first that $\kappa_{A}$ is strictly positive.
Define Gibbs potentials using the Möbius inversion~\eqref{eq:Moebiuspotentials}.  Note that
\begin{eqnarray*}         
  \sum_{\substack{C \subseteq A \\ | C |  \geq k}} (-1)^{| A \setminus C |}   \ln \kappa_C (x_C ; x_0 )       
        & = &  \sum_{\substack{C \subseteq A \\ |C| \geq k}} (-1)^{| A \setminus C |} \,  \frac{1}{{ \binom{|C|}{k} }}
                   \sum_{\substack{B \subseteq C \\ | B | = k }} \ln \kappa_C (x_C ; x_0 )     \\
        & = &  \sum_{\substack{C \subseteq A \\ C \in {\mathcal A}}} (-1)^{| A \setminus C |} \,  \frac{1}{{ \binom{|C|}{k} } }
                   \sum_{\substack{B \subseteq C \\ | B | = k }} \ln \kappa_B (x_B ; x_0 )     \\        
        & = &  \sum_{\substack{B \subseteq A \\ |B| = k}} \left\{ \sum_{R \subseteq A \setminus B } 
                   (-1)^{| A | - |R| - k} \,  \frac{1}{{ \binom{|R| + k}{k} } } \right\} \,  \ln \kappa_B (x_B ; x_0 )
\end{eqnarray*}
Together with~\eqref{eq:Moebiuspotentials} this gives
\begin{equation*}
  \phi_A (x_A, x_0) 
  \; = \; \sum_{\substack{C \subseteq A \\ | C | \leq k}}  \alpha_{A,C} \,  \ln \kappa_C(x_C ; x_0 ) \, ,
\end{equation*}
where
\begin{equation*}
  \alpha_{A,C} =
  \begin{cases}
    (-1)^{|A| - |C|}, & \text{ if } |C|< k \\
    \sum_{R \subseteq A \setminus C } 
                   (-1)^{| A | - |R| - k} \,  \frac{1}{{ \binom{|R| + k}{k} } }, &  \text{ if } |C|= k
  \end{cases}
\end{equation*}
depends only on the cardinalities of $A$ and $C$.  The statement follows with the choice
$\Psi_{C,A}(x_{C};x_{0})=\alpha_{A,C}\ln\kappa_{C}(x_{C};x_{0})$.

If $(\kappa_{A})$ is not strictly positive, then define $\lambda_{A}(x_{A};x_{0}) = \frac{1}{d_{0}}$ for all
$A\subseteq[n]$.  Then the functional modalities $(\lambda_{A})$ are $\Rcal_{k}$-robust for all $x\in\Scal$, and so
are the strictly positive functional modalities $(\kappa^{\epsilon}_{A})$ defined via $\kappa^{\epsilon}_{A} =
(1-\epsilon)\kappa_{A} + \epsilon\lambda_{A}$.  The statement follows from $\lim_{\epsilon\to
  0}\kappa^{\epsilon}_{A}=\kappa_{A}$.
\end{proof}

\begin{ex}
  Consider the case of $n=2$ binary inputs, $\Xcal_{1}=\Xcal_{2}=\{0,1\}$, and let $\Scal=\{(0,0),(1,1)\}$.  Then
  $\Rcal_{1}$-robustness on $\Scal$ means
  \begin{equation*}
    \kappa_{\{1\}}(x_{1}; x_{0}) = \kappa_{\{1,2\}}(x_{1},x_{2}; x_{0}) = \kappa_{\{2\}}(x_{2}; x_{0})
  \end{equation*}
  for all $x_{0}$ whenever $x_{1}=x_{2}$.  By Proposition~\ref{prop:robkernelcond} this translates into the conditions
  \begin{equation}
    \label{eq:ex3:potcond}
    \phi_{\{1,2\}}(x_{1},x_{2}; x_{0}) + \phi_{\{1\}}(x_{1}; x_{0})
    = 0 =
    \phi_{\{1,2\}}(x_{1},x_{2}; x_{0}) + \phi_{\{2\}}(x_{2}; x_{0})
  \end{equation}
  for all $x_{0}$ whenever $x_{1}=x_{2}$
  for the potentials $(\phi_{A})$ defined via~\eqref{eq:Moebiuspotentials}.  This means: Assuming that
  $(\kappa_{A})$ is $\Rcal_{1}$-robust, it suffices to specify the four functions
  \begin{align*}
    \phi_{\emptyset}(x_{0}),
    \phi_{\{1\}}(x_{1}; x_{0}),
    \phi_{\{1,2\}}(0,1; x_{0}),
    \phi_{\{1,2\}}(1,0; x_{0}).
  \end{align*}
  The remaining potentials can be deduced from~\eqref{eq:ex3:potcond}.  If only the values of $(\kappa_{A})$ for $x\in\Scal$ are needed, then it suffices to specify
$    \phi_{\emptyset}(x_{0})
    \text{ and }
    \phi_{\{1\}}(x_{1}; x_{0})$.
\end{ex}

  Even though the families $K_{k}$ and $\tilde K_{k}$ do not depend on the set $\Scal$, the choice of the set $\Scal$
  is essential: If the set $\Scal$ is too large, then the conditions~\eqref{invar} imply that the output $X_{0}$ is
  (unconditionally) independent of all inputs.  The theory developed in Sections~\ref{sec:robustness-and-CI}
  to~\ref{sec:k-robustness} discusses the constraints on conditionals imposed by the choice of $\Scal$.
  In particular, Section~\ref{sec:robust_functions} gives bounds on the strength of the interaction between the input
  nodes and the output node for given $\Rcal$ and $\Scal$.

  On the other hand, since $K_{k}$ and $\tilde K_{k}$ are independent of $\Scal$,
  Proposition~\ref{prop:rob-kernel-k-interactions} shows that these two families can be used to construct robust
  systems, when the input distribution $\pin$ is not known a priori (or may change over time) but must be learned by
  the system.

\section{Robustness and conditional independence}
\label{sec:robustness-and-CI}

We now study robustness of the joint distribution $p$ of $(X_{0},X_{[n]})$.  As stated in the introduction,
$p$ is called \emph{robust against knockout of $S$} if it satisfies
$\CI{X_{0}}{X_{S}}[X_{R}]$, where $R=[n]\setminus S$.  By definition this means that
\begin{equation}
  \label{eq:elementary-CI}
  p(x_{0},x_{S},x_{R})p(x_{0}',x_{S}',x_{R}) = p(x_{0},x_{S}',x_{R})p(x_{0}',x_{S},x_{R}),
\end{equation}
for all $x_{0},x_{0}'\in\Xcal_{0}, x_{S},x_{S}'\in\Xcal_{S}$ and $x_{R}\in\Xcal_{R}$.  Here, $p(x_{0},x_{S},x_{R})$ is an
abbreviation of $p(X_{0}=x_{0},X_{S}=x_{S},X_{R}=x_{R})$.  It is not difficult to see that this definition is equivalent
to the usual definition of conditional independence~\cite{DrtonSturmfelsSullivant09:Algebraic_Statistics}.  This
algebraic formulation makes it possible to study conditional independence with algebraic tools.

In order to formulate the results in higher generality, we will also consider CI statements of the form
$\CI{X_{0}}{X_{S}}[X_{R}=y]$ for some $S\subseteq[n]$, $R=[n]\setminus S$ and $y\in\Xcal_{R}$.  By definition, this is
equivalent to equations~\eqref{eq:elementary-CI} for all $x_{0},x_{0}'\in\Xcal_{0}, x_{S},x_{S}'\in\Xcal_{S}$ and
$x_{R}=y$.  Such a statement models the case that, if the value of the input variables $\Xcal_{R}$ is~$y$, then the
system does not need to know the remaining variables $X_{S}$ in order to compute its output.
Such CI statements naturally generalize canalizing~\cite{Kauffman93:Origins_of_Order} or nested
canalizing functions~\cite{JarrahRaposaLaubenbacher07:Nested_Canalyzing_and_other_Functions}, which have been studied in the context of
robustness.
The simpler statement $\CI{X_{0}}{X_{S}}[X_{R}]$ corresponds to the special case where
$\CI{X_{0}}{X_{S}}[X_{R}=y]$ for all $y\in\Xcal_{R}$.

Let $\Rcal$ be a collection of pairs $(R,y)$, where $R\subseteq[n]$ and $y\in\Xcal_{R}$.  Such a
collection will be called a \emph{robustness specification} in the following.  A joint distribution is called
\emph{$\Rcal$-robust} if it satisfies all conditional independence (CI) statements
\begin{equation}
  \label{eq:CI-collection}
  \CI{X_{0}}{X_{[n]\setminus R}}[X_{R}=y]
\end{equation}
for all $(R,y)\in\Rcal$.
We denote $\Pcal_{\Rcal}$ the set of all $\Rcal$-robust probability distributions.

\begin{ex}
  As before, let $\Rcal_{k}$ be the set of subsets of $[n]$ of cardinality $k$ or greater.  In other words, a
  probability measure $p$ is $\Rcal_{k}$-robust, if we can knock out any $n-k$ input variables without losing information
  on the output.
\end{ex}

Equations~\eqref{eq:elementary-CI} are polynomial equations in the elementary probabilities.  They are related to the
\emph{binomial edge ideals} introduced in~\cite{HHHKR10:Binomial_Edge_Ideals}.  The generalized binomial edge ideals
will be studied in Section~\ref{sec:gen-binom-edge-ideals}.  Here, we interpret the algebraic results from the point of
view of robustness.

Let $\Xcal=\Xcal_{1}\times\dots\times\Xcal_{n}$.  A robustness specification $\Rcal$ induces a graph $G_{\Rcal}$ on
$\Xcal$, where $x,x'\in\Xcal$ are connected by an edge if and only if there exists $(R,y)\in\Rcal$ such
that the restrictions of $x$ and $x'$ to $R$ satisfy $x|_{R}=x'|_{R}=y$.
\begin{defi}
  \label{def:robustness-structure}
  Let $\Ycal\subseteq\Xcal$, and denote $G_{\Rcal,\Ycal}$ the subgraph of $G_{\Rcal}$ induced by $\Ycal$.  The set
  $\Ycal$ is called \emph{$\Rcal$-connected} if $G_{\Rcal,\Ycal}$ is connected.  The set of connected components of
  $G_{\Rcal,\Ycal}$ is called a \emph{$\Rcal$-robustness structure}.  An $\Rcal$-robustness structure $\Bbf$ is
  \emph{maximal} if and only if $\cup\Bbf:=\cup_{\Zcal\in\Bbf}\Zcal$ satisfies any of the following equivalent
  conditions:
  \begin{enumerate}
  \item For any $x\in\Xcal\setminus\cup\Bbf$ there are edges $(x,y)$, $(x,z)$ in $G_{\Rcal}$ such that
    $y,z\in\cup\Bbf$ are not connected in $G_{\Rcal,\cup\Bbf}$.
  \item For any $x\in\Xcal\setminus\cup\Bbf$ the induced subgraph $G_{\Rcal,\cup\Bbf\cup\{x\}}$ has less connected
    components than $G_{\Rcal,\cup\Bbf}$.
  \end{enumerate}
\end{defi}

For any probability distribution $p$ on $\Xcal$, $x_{0}\in\Xcal_{0}$ and $x\in\Xcal$ denote $\tilde p_{x}$ the vector
with components $\tilde p_{x}(x_{0}) = p(X_{0}=x_{0},X_{[n]}=x)$.  Denote $\supp\tilde p:=\{x\in\Xcal:\tilde p_{x}\neq
0\}$.  For any family $\Bbf$ of subsets of $\Xcal$ let $\Pcal_{\Bbf}$ be the set of probability distributions $p$ that
satisfy the following two conditions:
\begin{enumerate}
\item $\supp\tilde p = \cup\Bbf$,
\item $\tilde p_{x}$ and $\tilde p_{y}$ are proportional, whenever there exists $\Zcal\in\Bbf$ such that $x,y\in\Zcal$.
\end{enumerate}
It follows from~\eqref{eq:decomposition-VG} and Theorem~\ref{thm:primary-decomposition} that $\Pcal_{\Rcal}$ equals the
disjoint union $\cup_{\Bbf}\Pcal_{\Bbf}$, where the union is over all 
$\Rcal$-robustness structures.  Alternatively, $\Pcal_{\Rcal}$ equals the union $\cup_{\Bbf}\ol{\Pcal_{\Bbf}}$, where
the union is over all maximal $\Rcal$-robustness structures.

For any $x\in\Xcal$ the vector $\tilde p_{x}$ is proportional to the conditional probability distribution
$P(\cdot|X_{[n]}=x)$ of $X_{0}$ given that $X_{[n]}=x$.  Hence:
\begin{lemma}
  \label{lem:Crobustcrit}
  Let $p$ be a probability distribution, and let $\Bbf$ be the set of connected components of $G_{\Rcal,\supp\tilde p}$.
  Then $p$ is $\Rcal$-robust if and only if $P(\cdot|X_{[n]}=x)=P(\cdot|X_{[n]}=y)$ whenever there exists $\Zcal\in\Bbf$
  such that $x,y\in\Zcal$.
\end{lemma}

The following lemma sheds light on the structure of $\Pcal_{\Bbf}$:
\begin{lemma}
  \label{lem:genconst}
  Fix an $\Rcal$-robustness structure $\mathbf{B}$.  Then $\Pcal_{\Bbf}$ consists of all probability measures of
  the form
  \begin{equation}
    \label{eq:genconst}
    p_{x_{0}x} =
    \begin{cases}
      \mu(\Zcal)\lambda_{\Zcal}(x) p_{\Zcal}(x_{0}) & \text{if }x\in \Zcal\in\Bbf,
      \\
      0                       & \text{ if }x\in\Xcal\setminus\cup\Bbf,
    \end{cases}
  \end{equation}
  where $\mu$ is a probability distribution on $\Bbf$ and $\lambda_{\Zcal}$ is a probability distribution on $\Zcal$ for
  each $\Zcal\in\Bbf$ and $(p_{\Zcal})_{\Zcal\in\Bbf}$ is a family of probability distributions on $\Xcal_{0}$.
\end{lemma}
\begin{proof}
  It is easy to see that $p$ is indeed a probability distribution.  By Lemma \ref{lem:Crobustcrit} it belongs to
  $\Pcal_{\Bbf}$.  In the other direction, any probability measure 
  can be written as a product
  \begin{equation*}
    p(x_{0},x_{1},\dots,x_{n}) = p(\Zcal) p\left(x_{1},\dots,x_{n}\middle|(X_{1},\dots,X_{n})\in \Zcal\right) p(x_{0}|x_{1},\dots,x_{n}),
  \end{equation*}
  if $(x_{1},\dots,x_{n})\in \Zcal\in\mathbf{B}$, and if $p$ is an $\Rcal$-robust probability distribution, then
  $p_{\Zcal}(x_{0}):=p(x_{0}|x_{1},\dots,x_{n})$ depends only on the block $\Zcal$ in which $(x_{1},\dots,x_{n})$ lies.
\end{proof}

\section{Robust functions}
\label{sec:robust_functions}

The factorization in Lemma~\ref{lem:genconst} admits the following interpretation:
\begin{prop}
  Let $\mathbf{B}$ be an $\Rcal$-robustness structure.  Then the set $\overline{\mathcal{R}}_{B}$ is the set of probability
  distributions such that
  \begin{equation*}
    p\big(X_{[n]}\in\cup\Bbf\big) = 1
  \end{equation*}
  and
  \begin{equation*}
    \CI{X_{0}}{X_{[n]}}[X_{[n]}\in\Zcal]
    \qquad\text{ for all }\Zcal\in\mathbf{B}.
  \end{equation*}
\end{prop}
In other words, the sets $\Zcal\in\mathbf{B}$ determine a partition of the set $\supp\tilde p$, which consists of all
outcomes of $X_{[n]}$ with non-zero probability under $p$.  Within each block $\Zcal$ the value of $X_{0}$ is
independent of all inputs.
Let $R\subseteq[n]$, and let $x,x'\in\Xcal_{[n]}$ satisfying $(R,x|_{R})\in\Rcal$ and
$(R,x'|_{R})\in\Rcal$.  If $x$ and $x'$ belong to different blocks in $\Bbf$, then $x|_{R}\neq x'|_{R}$.
Therefore, the knowledge of the input variables in $R$ is sufficient to determine in which block
$\Zcal\in\mathbf{B}$ we are.

When $p$ or $\Bbf$ is fixed we can introduce an additional random variable $B$ that takes values in $\mathbf{B}$.  The
situation is illustrated by the following graph:
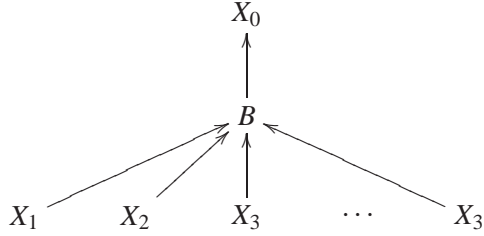
\begin{equation*}
      \xymatrix{
         && X_{0} && \\
         && B\ar[u] && \\
        X_{1}\ar[urr] & X_{2}\ar[ur] & X_{3}\ar[u] & \cdots & X_{3}\ar[ull]\\
     }
\end{equation*}
The arrows from the input variables $X_{1},\dots,X_{n}$ to $B$ are, in fact, deterministic:
\begin{equation*}
  B(x) = \Zcal \qquad\text{ if }x\in \Zcal\in\Bbf.
\end{equation*}
Note, however, that the function $B$ is only defined uniquely on $\cup\Bbf$, which is a set of measure one with respect
to $p$.  This means that in many cases it is enough to study robustness of functions on $\Xcal$.

\begin{defi}
  A function $f$ 
  defined on a subset $\Scal\subseteq\Xcal_{[n]}$ is $\Rcal$-robust if there exists an $\Rcal$-robustness
  structure $\mathbf{B}$ such that $\Scal=\cup\mathbf{B}$ and $f$ is constant on each $B\in\mathbf{B}$.
\end{defi}
There are two motivations for looking at this kind of functions: First, they occur in the special case of
$\Rcal$-robust probability distributions $p(X_{0},X_{1},\dots,X_{n})$ such that all conditional probability
distributions $p(X_{0}|x_{1},\dots,x_{n})$ are Dirac measure.  Second, as motivated above, we can associate to any
$\Rcal$-robust probability distribution $p$ a corresponding function $f$ characterizing the
$\Rcal$-robustness structure.  In order to reconstruct $p$ it is enough to specify the input distribution
$\pin(X_{1},\dots,X_{n})$ and a set of output distributions
$\left\{p(X_{0}|(X_{1},\dots,X_{n})\in\Zcal)\right\}_{\Zcal\in\mathbf{B}}$ in addition to the function
\mbox{$f:\Scal\to\mathbf{B}$}.
Note that natural examples of robust functions arise from the study of canalizing
functions~\cite{Kauffman93:Origins_of_Order,JarrahRaposaLaubenbacher07:Nested_Canalyzing_and_other_Functions}.

It is natural to ask the following question:
Given a certain robustness structure, how much freedom is left to choose a robust function $f$?  More precisely, how
large can the image of $f$ be?  Equivalently, how many components can an $\Rcal$-robustness structure $\Bbf$ have?

\begin{lemma}
  Let $f$ be an $\Rcal$-robust function.  The cardinality of the image of $f$ is bounded from above by
  \begin{equation*}
    \min\left\{\prod_{i\in R}d_{i}: (R,y)\in \Rcal\text{ for all }y\in \Xcal_{R} \right\}.
  \end{equation*}
\end{lemma}
\begin{proof}
  Suppose without loss of generality that $(\{1,\dots,r\},y)\in\Rcal$ for all $y\in\Xcal_{[r]}$ and that $d_{1}\dots
  d_{r}$ equals the above minimum.  
  The image of $f$ cannot be larger than $d_{1}\dots d_{r}$, since if we knock out all $X_{i}$ for $i>r$, then we can
  only determine $d_{1}\dots d_{r}$ states.
\end{proof}

\begin{ex}
  Suppose that $\Scal=\Xcal$.  This means that the $\Rcal$-robustness structure satisfies $\cup\Bbf = \Xcal$.  We first
  consider the case that $G_{\Rcal}$ is connected.  This is fulfilled, for example, if for any $k\in[n]$ there exists
  $R\subseteq[n]$ such that $k\notin S$ and $(R,y)\in\Rcal$ for all $y\in\Xcal_{R}$.  In this case an $\Rcal$-robust
  function $f$ takes only one value.

  Assume that $(R,y)\in\Rcal$ implies $(R,y')\in\Rcal$ for all $y'\in\Xcal_{R}$.  If $G_{\Rcal}$ is not connected,
  then
  some input variables may never be knocked out.  Let $T$ be the set of these input variables.  For every fixed value of
  $X_{T}$
  the function $f$ must be constant.  This means that $f$ can have
  $\prod_{i\in[n]\setminus T}d_{i}$ different values.
\end{ex}

\begin{rem}[Relation to coding theory]
  We can interpret $\Xcal$ as a set of words over the alphabet $[d_{m}]$ of length $n$, where $d_{m}=\max\{d_{i}\}$.
  For simplicity 
  assume that all $d_{i}$ are equal.
  Consider the uniform case $\Rcal=\Rcal_{k}$.  Then the task is to find a collection of subsets such that
  any two different subsets have Hamming distance at least $k$.  A related problem appears in coding theory: A code is a
  subset $\Ycal$ of $\Xcal$ and corresponds to the case that each element of $\Bbf$ is a singleton.  If
  distinct elements of the code have Hamming distance at least $n-k$, then a message can be reliably decoded even if only
  $k$ letters are transmitted correctly.
\end{rem}

\section{$\Rcal_{k}$-robustness}
\label{sec:k-robustness}

In this section we consider the symmetric case~$\Rcal=\Rcal_{k}$.
We fix $n$ and replace any prefix or subscript $\Rcal$ by $k$.

Let $k=0$.  Any pair $(x,y)$ is an edge in $G_{0}$.  This means that $\mathbf{B}$ can contain only one set $B$.  There
is only one maximal $0$-robustness structure, namely $\overline{\mathbf{B}}=\{\Xcal_{[n]}\}$.  The set $\mathcal{R}_{0}$
is irreducible.  This corresponds to the fact that $\Pcal_{n}$ is defined by $\CI{X_{0}}{X_{[n]}}$.

$\overline{\mathbf{B}}$ is actually a maximal $k$-robustness structure for any $0\le k\le n$.  This illustrates the fact
that the single CI statement $\CI{X_{0}}{X_{[n]}}$ implies all other CI statements of the
form~\eqref{eq:CI-collection}.  The corresponding set $\Pcal_{\overline{\mathbf{B}}}$ contains all probability
distributions of $\Pcal_{0}$ of full support.

Now let $k=1$.  In the case $n=2$, we obtain results by Alexander Fink, which can be reformulated as
follows~\cite{Fink11:Binomial_ideal_of_intersection_axiom}:
%
\emph{  Let $n=2$.  A $1$-robustness structure $\mathbf{B}$ is maximal if and only if the following statements hold:
  \begin{itemize}
  \item Each $B\in\mathbf{B}$ is of the form $B=S_{1}\times S_{2}$, where $S_{1}\subseteq \Xcal_{1}, S_{2}\subseteq\Xcal_{2}$.
  \item For every $x_{1}\in\Xcal_{1}$ there exists $B\in\mathbf{B}$ and $x_{2}\in\Xcal_{2}$ such that $(x_{1},x_{2})\in B$,
    and conversely.
  \end{itemize}
}

In~\cite{Fink11:Binomial_ideal_of_intersection_axiom} a different description is given: The block $S_{1}\times S_{2}$ can
be identified with the complete bipartite graph on $S_{1}$ and $S_{2}$.  In this way, every maximal $1$-robustness
structure corresponds to a collection of complete bipartite subgraphs with vertices in $\Xcal_{1}\cup\Xcal_{2}$ such that
every vertex in $\Xcal_{1}$ resp.~$\Xcal_{2}$ is part of one such subgraph.


This result generalizes in the following way:
\begin{lemma}
  A $1$-robustness structure $\mathbf{B}$ is maximal if and only if the following statements hold:
  \begin{itemize}
  \item Each $B\in\mathbf{B}$ is of the form $B=S_{1}\times\dots\times S_{n}$, where $S_{i}\subseteq\Xcal_{i}$.
  \item Fix $j\in[n]$ and $x_{i}\in\Xcal_{i}$ for all $i\in[n]$, $i\neq j$.  Then there exist $x_{j}\in\Xcal_{j}$ such that
    $(x_{1},\dots,x_{n})\in\cup_{B\in\mathbf{B}}B$.  In other words, whenever $n-1$ components of $(x_{1},\dots,x_{n})$
    are prescribed, there exist an $n$-th component such that $(x_{1},\dots,x_{n})\in\cup_{B\in\mathbf{B}}B$.
  \end{itemize}
\end{lemma}
\begin{proof}
  We say that a subset $\Ycal$ of $\Xcal$ is connected if $G_{\Rcal,\Ycal}$ is connected.  Suppose
  that $\mathbf{B}$ is maximal. %
  Let $B\in\mathbf{B}$ and let $S_{i}$ be the projection of $B\subset\Xcal_{[n]}$ to $\Xcal_{i}$.  Let
  $B'=S_{1}\times\dots\times S_{n}$.  Then $B\subseteq B'$.  We claim that $(\mathbf{B}\setminus\{B\})\cup\{B'\}$ is
  another coarser $1$-robustness structure.  By Definition~\ref{def:robustness-structure} we need to show that $B'$
  is connected and that $A\cup B'$ is not connected for all $A\in\mathbf{B}\setminus\{B\}$.  The first condition follows
  from the fact that $B$ is connected.
%
  For the second condition assume to the contrary that there are $x\in B'$ and $y\in A$ such that
  $x=(x_{1},\dots,x_{n})$ and $y = (y_{1},\dots,y_{n})$ disagree in at most $n-1$ components.  Then there exists a
  common component $x_{l}=y_{l}$.  By construction there exists $z=(z_{1},\dots,z_{n})\in B$ such that
  $z_{l}=y_{l}=x_{l}$, hence $A\cup B$ is connected, in contradiction to the assumptions.  This shows that each $B$ has
  a product structure.

  Write $B=S^{B}_{1}\times\dots\times S^{B}_{n}$ for each $B\in\mathbf{B}$.  Obviously $S^{B}_{i}\cap
  S^{B'}_{i}=\emptyset$ for all $i\in[n]$ and all $B,B'\in\mathbf{B}$ if $B\neq B'$.  The second assertion claims that
  $\cup_{B\in\mathbf{B}}S^{B}_{i}=\Xcal_{i}$ for all $i\in[n]$: Assume to the contrary that $l\in\Xcal_{i}$ is contained in
  no $S^{B}_{i}$.  Take any $B$ and define $B':= S^{B}_{1}\times\dots\times (S^{B}_{i}\cup\{l\})\times\dots\times
  S^{B}_{n}$.  Then $\left(\mathbf{B}\setminus B\right)\cup\{B'\}$ is a coarser $1$-robustness structure.

  Now assume that $\mathbf{B}$ is a $1$-robustness structure satisfying the two assertions of the theorem.  For any
  $x\in\Xcal\setminus\cup\mathbf{B}$ there exists $y\in\cup\Bbf$ such that $x_{1}=y_{1}$, and hence $(x,y)$ is an edge
  in $G_{1}$.
  This implies maximality.
\end{proof}

The last result can be reformulated in terms of $n$-partite graphs
generalizing~\cite{Fink11:Binomial_ideal_of_intersection_axiom}: Namely, the $1$-robustness structures are in
one-to-one relation with the $n$-partite subgraphs of $K_{d_{1},\dots,d_{n}}$ such that every connected component is itself a
complete $n$-partite subgraph $K_{e_{1},\dots,e_{n}}$ with $e_{i}>0$ for all $i\in[n]$.  Here, an $n$-partite subgraph
is a graph which can be coloured by $n$ colours such that no two vertices with the same colour are connected by an edge.

Unfortunately the nice product form of the maximal $1$-robustness structures does not generalize to $k>1$:
\begin{ex}[Binary three inputs]
  If $n=3$ and $d_{1}=d_{2}=d_{3}=2$ and $k=2$, then the graph $G_{\Rcal}$ is the graph of the cube. For a maximal 1-robustness structure $\Bbf$ the
  set $\Xcal\setminus\cup\Bbf$ can be any one of
  the following: 
  \begin{itemize}
  \item The empty set
  \item A set of cardinality 4 corresponding to a plane leaving two connected components of size 2
  \item A set of cardinality 4 containing all vertices with the same parity.
  \item A set of cardinality 3 cutting off a vertex.
  \end{itemize}
  An example for the last case is
  \begin{equation*}
    \mathbf{B} := \left\{\{(1,1,1)\}, \{(2,2,2),(2,2,1),(2,1,2),(1,2,2)\}\right\}.
  \end{equation*}
  Only the isolated vertex has a product structure.
\end{ex}

Generically, the smaller $k$, the easier it is to describe the structure of all $k$-robustness structures.  We have seen
above that the cases $k=0$ and $k=1$ are particularly nice.  One might expect that all $k$-robustness
structures are also $(k+1)$-robustness structures for all $k$.
Unfortunately, this is not true in general:
\begin{ex}
  Consider $n=4$ binary random variables $X_{1},\dots,X_{4}$.  Then
  \begin{equation*}
    \mathbf{B} := \left\{ \{(1,1,1,1),(2,2,1,1)\}, \{(1,2,2,2), (2,1,2,2)\} \right\}
  \end{equation*}
  is a maximal $2$-robustness structure.  Both elements of $\mathbf{B}$ are $\sim_{2}$-connected, but not
  $\sim_{3}$-connected.
\end{ex}

The following two lemmas relate $k$-robustness to $l$ robustness for $l>k$:
\begin{lemma}
  Let $\mathbf{B}$ be a $k$-robustness structure.  For every $l > k$ there exists an $l$-robustness structure $\Bbf'$
  such that the following holds: For any $\Ycal\in\Bbf$ there exists precisely one $\Ycal'\in\Bbf'$ such that
  $\Ycal\subseteq\Ycal'$.
\end{lemma}
\begin{proof}
  The statements \eqref{eq:CI-collection} for $k$ imply the same statements of $l$, so
  $\overline{\Pcal_{\mathbf{B}}}$ is a closed subset of $\Pcal_{l}$.  Thus
  $\overline{\Pcal_{\mathbf{B}}}$ lies in one irreducible subset $\overline{\Pcal_{\mathbf{B}'}}$ of
  $\Pcal_{l}$.  The statement now follows from Lemma~\ref{lem:VGYcontainsVGZ}.
\end{proof}

\begin{lemma}
  \label{lem:smallk}
  Assume that $d_{1}=\dots=d_{n}=2$, and let $\mathbf{B}$ be a maximal $k$-robustness structure of binary random
  variables.  Then each $B\in\mathbf{B}$ is connected as a subset of $G_{s}$ for all $s\le n - 2 k$.
\end{lemma}
\begin{proof}
  We can identify elements of $\Xcal$ with $01$-strings of length $n$.  Denote $I_{r}$ the string $1\dots10\dots0$ of
  $r$ ones and $n-r$ zeroes in this order.  Without loss of generality assume that $I_{0}, I_{l}$ are two elements of
  $B\in\mathbf{B}$, where $k \ge n-l < s$.  
  Let $m = \lceil\frac{l}{2}\rceil$ and consider $I_{m}$.  We want to prove that we can replace $B$ by $B\cup\{I_{m}\}$
  and obtain another, coarser $k$-robustness structure.  By maximality this will imply that $I_{0}$ and
  $I_{l}$ are indeed connected by a path in $G_{s}$.

  Otherwise there exists $A\in\mathbf{B}$ and $x\in A$ such that $x$ and $I_{m}$ agree in at least $k$ components.
  Let $a$ be the number of zeroes in the first $m$ components of $x$, let $b$ be the number of ones in the components
  from $m+1$ to $l$ and let $c$ be the number of ones in the last $n-l$ components.  Then $I_{m}$ and $x$ disagree in
  $a+b+c \le n-k$ components.  On the other hand, $x$ and $I_{0}$ disagree in $(m-a) + b + c$ components, and $x$ and
  $I_{l}$ disagree in $a + ((l-m)-b) + c \le a + (m-b) + c$ components.  Assume that $a\ge b$ (otherwise exchange
  $I_{0}$ and $I_{l}$).  Then $x$ and $I_{0}$ disagree in at most $m + c \le \lceil\frac{l}{2}\rceil + n - l = n -
  \lfloor\frac{l}{2}\rfloor \le n-k$ components, so $A\cup B$ is connected, in contradiction to the assumptions.
\end{proof}

\section{Generalized binomial edge ideals}
\label{sec:gen-binom-edge-ideals}

We refer to~\cite{CoxLittleOShea08:Ideals_Varieties_Algorithms} for an introduction to the algebraic terminology that is
used in this section.

Let $\Xcal$ be a finite set, $d_{0}>1$ an integer, and denote $\tilde\Xcal=\Xcal_{0}\times\Xcal$.  Fix a field $\KK$.
Consider the polynomial ring $R = \KK[p_{x}:x\in\tilde\Xcal]$ with $|\tilde\Xcal|$ unknowns $p_{x}$ indexed by
$\tilde\Xcal$.  For all $i,j\in\Xcal_{0}$ and all $x,y\in\Xcal$ let
\begin{equation*}
  f^{ij}_{xy} = p_{ix} p_{jy} - p_{iy} p_{jx}.
\end{equation*}
For any graph $G$ on $\Xcal$ the ideal $I_{G}$ in $R$ generated by the binomials $f^{ij}_{xy}$ for all $i,j\in\Xcal_{0}$
and all edges $(x,y)$ in $G$ is called the $d_{0}$th \emph{binomial edge ideal} of $G$ over $\KK$.  This is a direct
generalization of~\cite{HHHKR10:Binomial_Edge_Ideals} and~\cite{Ohtani09:Ideals_of_some_2-minors}, where the same ideals
have been considered in the special case $d_{0}=2$.

Choose a total order $>$ on $\Xcal$ (e.g.~choose a bijection $\Xcal\cong[|\Xcal|]$).  This induces a lexicographic
monomial order, that will also be denoted by $>$, via
\begin{equation*}
  p_{ix}> p_{jy} \qquad\Longleftrightarrow\qquad
  \begin{cases}
    \text{ either } & i > j,\\
    \text{ or } & i = j \text{ and } x > y.
  \end{cases}
\end{equation*}
A Gröbner basis for $I_{G}$ with respect to this order can be constructed using the
following definitions:
\begin{defi}
  \label{def:admissible-path}
  A path $\pi:x=x_0,x_1,\ldots,x_r=y$ from $x$ to $y$ in $\Xcal$ is called \emph{admissible} if
  \begin{enumerate}
  \item[(i)] $x_k\neq x_\ell$ for $k\neq \ell$, and $x < y$;
  \item[(ii)] for each $k=1,\ldots,r-1$ either $x_k<x$ or $x_k>y$;
  \item[(iii)] for any proper subset $\{y_1,\ldots,y_s\}$ of $\{x_1,\ldots,x_{r-1}\}$, the sequence $x,y_1,\ldots,y_s,y$
    is not a path.
  \end{enumerate}
  A function $\kappa: \{0,\dots,r\}\to [d]$ is called \emph{$\pi$-antitone} if it satisfies
  \begin{equation}
    x_{s} < x_{t} \Longrightarrow \kappa(s) \ge \kappa(t),\text{ for all }1 \le s,t \le r.
  \end{equation}
  $\kappa$ is \emph{strictly $\pi$-antitone} if it is $\pi$-antitone and satisfies $\kappa(0) > \kappa(r)$.
\end{defi}

The notion of $\pi$-antitonicity also applies to paths which are not necessarily admissible.  However, since admissible
paths are \emph{injective} (i.e.~they only pass at most once at each vertex), we may write $\kappa(\ell)$ in the
admissible case, instead of $\kappa(s)$, if $\ell=\pi(s)$.

For any $x<y$, any admissible path
$\pi: x=x_0,x_1,\ldots,x_r=y$
from $x$ to $y$ and any $\pi$-antitone function $\kappa$ associate the monomial
\begin{equation*}
  u_{\pi}^{\kappa}= \prod_{k=1}^{r-1}p_{\kappa(k)x_k}.
\end{equation*}

\begin{thm}
  \label{thm:Gbasis}
  The set of binomials
  \begin{multline*}
    {\mathcal G}
    = \smash{\bigcup_{i<j} \,\bigg\{}
    \,u_{\pi}^{\kappa}f_{xy}^{\kappa(y)\kappa(x)}\,:\; x<y,\; \pi \text{ is an admissible path in }G\text{ from $x$ to $y$},
      \\
      \kappa\text{ is strictly $\pi$-antitone}\,\smash[t]{\bigg\}}
  \end{multline*}
  is a reduced Gröbner basis of $I_G$ with respect to the monomial order introduced above.
\end{thm}

The proof makes use of the following lemma, which explains $\pi$-antitonicity:
\begin{lemma}
  \label{lem:pi-antiton}
  Let $\pi:x_{0},\dots,x_{r}$ be a path in $G$, and let $\kappa:\{0,\dots,r\}\to [d]$ be an arbitrary function.  If
  $\kappa$ is not $\pi$-antitone, then there exists $g\in\mathcal{G}$ such that $\ini(g)$ divides the monomial
  $u_{\pi}^{\kappa}= \prod_{k=1}^{r-1}p_{\kappa(k)x_k}$.
\end{lemma}
\begin{proof}
  Let $\tau:y_{0},\dots,y_{s}$ be a minimal subpath of $\pi$ with respect to the property that the restriction of
  $\kappa$ to $\tau$ is not $\tau$-antitone.  This means that $\kappa$ is $\tau_{0}$-antitone and $\tau_{s}$-antitone,
  where $\tau_{0}=y_{1},\dots,y_{s}$ and $\tau_{s}=y_{0},\dots,y_{s-1}$.  Assume without loss of generality that
  $y_{0}<y_{s}$, otherwise reverse $\tau$.  The minimality implies that $\kappa(y_{0})<\kappa(y_{s})$.  It follows that
  $\tau$ is admissible: By minimality, if $y_{0}<y_{k}<y_{s}$, then $\kappa(y_{k}) \ge \kappa(y_{s}) > \kappa(y_{0}) \ge
  \kappa(y_{k})$, a contradiction.  Define
  \begin{equation*}
    \kappa(k) =
    \begin{cases}
      \kappa(s), & \text{ if }k=0, \\
      \kappa(0), & \text{ if }k=s, \\
      \kappa(k), & \text{ if }0<k<s.
    \end{cases}
  \end{equation*}
  Then $\kappa$ is $\tau$-antitone, and
  $\ini(u_{\tau}^{\kappa}f_{y_{0}y_{s}}^{\kappa(y_{s})\kappa(y_{0})})$ divides $u_{\pi}^{\kappa}$.
\end{proof}

\begin{proof}[Proof of Theorem~\ref{thm:Gbasis}]
  The proof is organized in three steps.

  \medskip
  \noindent
  {\em Step 1: ${\mathcal G}$ is a subset of $I_G$.}
  Let $\pi: x=x_0,x_1,\ldots,x_{r-1},x_r=y$ be an admissible path in $G$.  We show that $u_\pi^{\kappa}
  f_{xy}^{\kappa(j)\kappa(i)}$ belongs to $I_G$ using induction on $r$.  Clearly the assertion is true if $r = 1$, so
  assume $r > 1$.  Let $A = \{ x_k : x_k < x \}$ and $B = \{ x_\ell : x_\ell > y \}$.  Then either $A \neq \emptyset$ or
  $B \neq \emptyset$.

  Suppose $A \neq \emptyset$ and set $x_{k} = \max A$.  The two paths $\pi_1 : x_{k}, x_{k-1}, \ldots, x_1, x_0=x$
  and $\pi_2 : x_{k}, x_{k+1}, \ldots, x_{r-1}, x_r = y$ in $G$ are admissible.
  Let $\kappa_{1}$ and $\kappa_{2}$ be the restrictions of $\kappa$ to $\pi_{1}$ and $\pi_{2}$.  Let $a=\kappa(r)$,
  $b=\kappa(0)$ and $c=\kappa(k)$.  The calculation
  \begin{multline*}
    (p_{by}p_{ax}-p_{bx}p_{ay})p_{cx_{k}}
    \\
    = (p_{cx}p_{bx_{k}} - p_{cx_{k}}p_{bx})p_{ay}
    - (p_{cx}p_{ax_{k}} - p_{cx_{k}}p_{ax})p_{by}
    - p_{cx}(p_{bx_{k}}p_{ay}- p_{by}p_{ax_{k}})
  \end{multline*}
  implies that $u_{\pi}^{\kappa}f_{xy}^{ab}$ lies in the ideal generated by $u_{\pi_{1}}^{\kappa_{1}}f_{xx_{k}}^{bc}$,
  $u_{\pi_{1}}^{\kappa_{1}}f_{xx_{k}}^{ac}$ and $u_{\pi_{2}}^{\kappa_{2}}f_{x_{k}y}^{ab}$.  By induction it lies in $I_{G}$.

  The case $B \neq \emptyset$ can be treated similarly.

  \medskip
  \noindent
  {\em Step 2: ${\mathcal G}$ is a Gröbner basis of $I_G$.}
  Let $\pi:x_0,\dots,x_{r}$ and $\sigma:y_{0},\dots,y_{s}$ be admissible paths in $G$ with $x_{0}<x_{r}$ and
  $y_{0}<y_{s}$, and let $\kappa$ and $\mu$ be $\pi$- and $\sigma$-antitone.  By Buchberger's criterion we need to show
  that the $S$-pairs $s := S(u_\pi^{\kappa} f_{x_{0}x_{r}}^{\kappa(r)\kappa(0)}, u_\sigma^{\mu} f_{y_{0}y_{s}}^{\mu(s)\mu(0})$ reduces
  to zero.

  If $S\neq 0$, then $S$ is a binomial.  Write $S=S_{1}-S_{2}$, where $S_{1}=\ini(S)$.
  $S$ is homogeneous with respect to the multidegrees given by
  \begin{equation*}
    \deg(p_{zm})_{b} = \delta_{zb} =
    \begin{cases}
      1, & \text{ if } z=b, \\
      0, & \text{ else.}
    \end{cases}
  \end{equation*}
  and
  \begin{equation*}
    \deg(p_{zm})_{n} = \delta_{mn} =
    \begin{cases}
      1, & \text{ if } m=n, \\
      0, & \text{ else.}
    \end{cases}
  \end{equation*}

  If $\pi$ and $\sigma$ are disjoint, then $S=0$, since $u_\pi^{\kappa} f_{x_{0}x_{r}}^{\kappa(r)\kappa(0)}$ and
  $u_\sigma^{\mu} f_{y_{0}y_{s}}^{\mu(s)\mu(0)}$ contain different variables.  The same happens if the intersection of
  $\pi$ and $\sigma$ does not involve the starting or end points of $\pi$ and $\sigma$, since in this case $S$ is
  proportional to the $S$-pair of the two monomials $u_\pi^{\kappa}$ and $u_\sigma^{\mu}$.

  Assume that $\pi$ and $\sigma$ meet and that $S\neq 0$.  Then $S_{1}$ and $S_{2}$ are monomials, and the unknowns
  $p_{ix}$ occurring in $S_{1}$ and $S_{2}$ satisfy $x\in\pi\cup\sigma$.
  Assume that there are $x<y$ such that $D_{x}:=\min\{i\in\Xcal_{0}: p_{ix}\,|\, S_{1}\} < \max\{i\in\Xcal_{0}:
  p_{iy}\,|\,S_{1}\}=:D_{y}$.  Since $\pi\cup\sigma$ is connected there is an injective path $\tau:z_{0},\dots,z_{s}$
  from $x=z_{0}$ to $y=z_{s}$ in $\pi\cup\sigma$.  Choose a map $\lambda:\{0,\dots,s\}$ such that $\lambda(0) = D_{x}$,
  $\lambda(s) = D_{y}$ and $p_{\lambda(a)a}\,|\,S_{1}$ for all $0\le a\le s$.  Then $u_{\tau}^{\lambda}$ divides
  $S_{1}$, and $\lambda$ is not $\tau$-antitone.  So we can apply Lemma~\ref{lem:pi-antiton} in order to reduce $S$ to a
  smaller binomial.

  Let $S'$ be the reduction of $S$ modulo $\Gcal$.  If $S'\neq 0$, then let $S'_{1}=\ini(S')$.  The above argument shows
  that $\min\{i\in\Xcal_{0}: p_{ix}\,|\, S'_{1}\} \ge \max\{i\in\Xcal_{0}: p_{iy}\,|\,S'_{1}\}$ for all $x<y$.  This
  property characterizes $S'_{1}$ as the unique minimal monomial in $R$ with multidegree $\deg(S'_{1}) = \deg(S)$.  But
  since the reduction algorithm turns binomials into binomials, $S'-S'_{1}$ is also a monomial of multidegree $\deg(S)$,
  and smaller than $\deg(S'_{1})$.  This contradiction shows $S'=0$.

  \medskip
  \noindent
  {\em Step 3: $\Gcal$ is reduced.}
  Let $\pi:x_0,\dots,x_{r}$ and $\sigma:y_{0},\dots,y_{s}$ be admissible paths in $G$ with $x_{0}<x_{r}$ and
  $y_{0}<y_{s}$, and let $\kappa$ and $\mu$ be $\pi$- and $\sigma$-antitone.
  Let $u=\kappa(r),v=\kappa(0),w=\mu(s),t=\mu(0)$, and suppose that $u_\pi^{\kappa} p_{ux_{0}}p_{vx_{r}}$ divides either
  $u_\sigma^{\mu} p_{wy_{0}}p_{ty_{s}}$ or $u_\sigma^{\mu} p_{wy_{s}} p_{ty_{0}}$.  Then $\{ x_0, \ldots, x_r \}$ is
  a 
  subset of $\{ y_0, \ldots, y_s \}$, and $\kappa(b) = \mu(\sigma^{-1}(x_{b}))$ for $0<b<r$.
  %

  If $x_{0}=y_{0}$ and $x_{r}=y_{s}$, then $\pi$ is a sub-path of $\sigma$.  By Definition~\ref{def:admissible-path},
  $\pi$ equals $\sigma$ (up to a possible change of direction).  Hence $u_\pi^{\kappa}
  f_{x_{0}x_{r}}^{\kappa(r)\kappa(0)}$ and $u_\sigma^{\mu} f_{y_{0}y_{s}}^{\mu(s)\mu(0)}$ have the same (total) degree,
  hence they agree.

  If $x_{0}=y_{0}$ and $x_{r}\neq y_{s}$, then $p_{vx_{r}}$ divides $u_{\sigma}^{\mu}$, and so $x_{r}=y_{t}$ for some
  $t<s$ such that $v=\mu(t)$.  Then $y_{t}=x_{r}>x_{0}=y_{0}$, and hence $v\le\mu(0)=\kappa(0)<\kappa(r)$, in
  contradiction to $x_{0}<x_{r}$.
  A similar argument
  applies if $x_{0}\neq y_{0}$ and $x_{r}=y_{s}$.
  Finally, if $x_{0}\neq y_{0}$ and $x_{r}\neq y_{s}$, then $p_{ux_{0}}p_{vx_{r}}$ divides $u_\sigma^{\mu}$.  This
  implies $u=\kappa(0)=\kappa(j)=v$, a contradiction.
\end{proof}

\begin{cor}
  \label{cor:radical}
  $I_G$ is a radical ideal.
\end{cor}

\begin{proof}
  The assertion follows from Theorem~\ref{thm:Gbasis} the following general fact: A graded ideal that has a Gröbner
  basis with square-free initial terms is radical.  See the proof of~\cite[Corollary~2.2]{HHHKR10:Binomial_Edge_Ideals} for the details.
\end{proof}

Since $I_{G}$ is radical, in order to compute the primary decomposition of the ideal it is enough to compute the minimal
primes.
We are mainly interested in the irreducible decomposition of the variety $V_{G}$ of $I_{G}$ in the case of
characteristic zero.  While the basic arguments remain true for finite base fields there is no relation between the
primary decomposition of an ideal and the irreducible decomposition of its variety, since the irreducible decomposition
consists of all closed points in this case.  The following definition is needed: Two vectors $v,w$ (living in the same
$\KK$-vector space) are \emph{proportional} whenever $v=\lambda w$ or $w=\lambda v$ for some $\lambda\in\KK$.  A set of
vectors is \emph{proportional} if each pair is proportional.  Since $\lambda=0$ is allowed, proportionality is not
transitive: If $v$ and $w$ are proportional and if $u$ and $v$ are proportional, then we can conclude that $u$ and $w$
must be proportional only if $v\neq 0$.

We now study the solution variety $V_{G}$ of $I_{G}$, which is a subset of $\KK^{\Xcal_{0}\times\Xcal}$.  As usual,
elements of $\KK^{\Xcal_{0}\times\Xcal}$ will be denoted with the same symbol $p=(p_{ix})_{i\in\Xcal_{0},x\in\Xcal}$ as the
unknowns in the polynomial ring $R = \KK[p_{ix}:(i,x)\in\Xcal_{0}\times\Xcal]$.  Such a $p$ can be written as a $d_{0}\times|\Xcal|$-matrix.  Each binomial
equation in $I_{G}$ imposes conditions on this matrix saying that certain submatrices have rank 1.  For a fixed edge
$(x,y)$ in $G$ the equations $f^{ij}_{xy}=0$ for all $i,j\in\Xcal_{0}$ require that the submatrix
$(p_{kz})_{k\in\Xcal_{0},z\in\{x,y\}}$ has rank one.  More generally, if $K\subset\Xcal$ is a clique (i.e.~a complete
subgraph), then the submatrix $(p_{kz})_{k\in\Xcal_{0},z\in K}$ has rank one.  This means that all columns of this
submatrix are proportional.  The columns of $p$ will be denoted by $\tilde p_{x}$, $x\in\Xcal$.  
  A point $p$ lies in $V_{G}$ if and only if $\tilde p_{x}$ and $\tilde p_{y}$ are proportional for all edges
  $(x,y)$ of $G$.

Even if the graph $G$ is connected, not all columns $\tilde p_{x}$ must be proportional to each other, since
proportionality is not a transitive relation.  Instead, there are ``blocks'' of columns such that all columns within one
block are proportional.

For any $p\in\KK^{\Xcal_{0}\times\Xcal}$ let $G_{p}$ be the subgraph of $G$ induced by
$\supp\tilde p:=\{x\in\Xcal:\tilde p_{x}\neq 0\}$.  
We have shown:
\begin{itemize}
\item A point $p$ lies in $V_{G}$ if and only if $\tilde p_{x}$ and $\tilde p_{y}$ are proportional whenever
  $x,y\in\supp\tilde p$ lie in the same connected component of $G_{p}$.
\end{itemize}

For any subset $\Ycal\subseteq\Xcal$ denote $G_{\Ycal}$ the subgraph of $G$ induced by $\Ycal$.
Let $V_{G,\Ycal}$ be the set of all $p\in\KK^{\Xcal_{0}\times\Xcal}$ for which
$\tilde p_{x}=0$ for all $x\in\Xcal\setminus\Ycal$ and for which $\tilde p_{x}$ and $\tilde p_{y}$ are
proportional whenever $x,y\in\Xcal$ lie in the same connected component of $G_{\Ycal}$.  Then
\begin{equation}
  \label{eq:decomposition-VG}
  V_{G} = \cup_{\Ycal\subseteq\Xcal} V_{G,\Ycal}.
\end{equation}
The sets $V_{G,\Ycal}$ are irreducible algebraic varieties:
\begin{lemma}
  \label{lem:VGY-is-sirreducible}
  For any $\Ycal\subseteq\Xcal$ the set $V_{G,\Ycal}$ is the variety of the ideal $I_{G,\Ycal}$
  generated by the monomials
  \begin{equation}
    \label{eq:monomials}
    p_{ix} \qquad\text{for all }x\in\Xcal\setminus\Ycal\text{ and }i\in\Xcal_{0},
  \end{equation}
  and the binomials $f_{xy}^{ij}$ for all $i,j\in\Xcal_{0}$ and all $x,y\in\Ycal$ that lie in the same connected
  component of $G_{\Ycal}$.  The ideal $I_{G,\Ycal}$ is prime.
\end{lemma}
\begin{proof}
  The first statement follows from the definition of $V_{G,\Ycal}$.  Write $I^{1}_{G,\Ycal}$ for the ideal generated by
  all monomials \eqref{eq:monomials}, and for any $\Zcal\subseteq\Ycal$ write $I^{2}_{\Zcal}$ for the ideal generated by
  the binomials $f_{xy}^{ij}$, with $i,j\in\Xcal_{0}$ and $x,y\in\Zcal$.  Then $I^{1}_{G,\Ycal}$ is obviously prime.
  Each of the $I^{2}_{\Zcal}$ is a $2\times2$ determinantal ideal.  It is a classical (but difficult) result that this
  ideal is the defining ideal of a Segre embedding, and that it is prime (see~\cite{Sturmfels91:GB_of_toric_varieties}
  for a rather modern proof).  The ideal $I_{G,\Ycal}$ is the sum of the prime ideal $I^{1}_{G,\Ycal}$ and the prime
  ideals $I^{2}_{\Zcal}$ for all connected components $\Zcal$ of $G_{\Ycal}$, and since the defining equations of all
  these ideals involve disjoint sets of unknowns, $I_{G,\Ycal}$ itself is prime.
\end{proof}
The decomposition~\eqref{eq:decomposition-VG} is not the irreducible decomposition of $V_{G}$, because the union is
redundant.  Let $\Ycal,\Zcal\subseteq\Xcal$.  Using Lemma~\ref{lem:VGY-is-sirreducible} it is easy to remove the
redundant components:
\begin{lemma}
  \label{lem:VGYcontainsVGZ}
  Let $\Ycal,\Zcal\subseteq\Xcal$.  Then $V_{G,\Ycal}$ contains $V_{G,\Zcal}$ if and only
  if the following two conditions are satisfied:
  \begin{itemize}
  \item $\Zcal\subseteq\Ycal$.
  \item If $x,y\in\Zcal$ are connected in $G_{\Ycal}$, then they are connected in $G_{\Zcal}$.
  \end{itemize}
\end{lemma}
\begin{proof}
  Assume that $V_{G,\Ycal}\subseteq V_{G,\Zcal}$.  Then $I_{G,\Ycal}\supseteq I_{G,\Zcal}$.  For
  any $x\in\Xcal\setminus\Zcal$ and any $i\in\Xcal_{0}$ this implies $p_{ix}\in I_{G,\Ycal}$.  On the
  other hand, Lemma~\ref{lem:VGY-is-sirreducible} shows that the point with coordinates
  \begin{equation*}
    p_{iy}=
    \begin{cases}
      1, & \qquad\text{if }y\in\Ycal,\\
      0, & \qquad\text{else},
    \end{cases}
  \end{equation*}
  lies in $V_{G,\Ycal}$, and hence in $V_{G,\Zcal}$.  This implies $x\in\Ycal$.

  Let $x\in\Zcal$.  
  Choose two linearly independent non-zero vectors $v,w\in\KK^{d_{0}}$.  By Lemma~\ref{lem:VGY-is-sirreducible} the
  matrix with columns
  \begin{equation*}
    \tilde p_{y}=
    \begin{cases}
      v, & \qquad\text{if }y\text{ is connected to }x\text{ in }G_{\Ycal},\\
      w, & \qquad\text{if }y\in\Ycal\text{ is not connected to }x\text{ in }G_{\Ycal},\\
      0, & \qquad\text{else},
    \end{cases}
  \end{equation*}
  is contained in $V_{G,\Ycal}$ and hence in $V_{G,\Zcal}$.  Therefore, if $z$ is connected to $x$ in
  $G_{\Ycal}$, then it is connected to $x$ in $G_{\Zcal}$.

  Conversely, if the two conditions are satisfied, then all defining equations of $I_{G,\Zcal}$ lie in $I_{G,\Ycal}$.
\end{proof}
\begin{thm}
  \label{thm:primary-decomposition}
  The primary decomposition of $V_{G}$ is
  \begin{equation*}
    I_{G} = \cap_{\Ycal} I_{G,\Ycal},
  \end{equation*}
  where the intersection is over all $\Ycal\subseteq\Xcal$ such that the following holds: For any
  $x\in\Xcal\setminus\Ycal$ there are edges $(x,y)$, $(x,z)$ in $G$ such that $y,z\in\Ycal$ are not
  connected in $G_{\Ycal}$.  Equivalently, for any $x\in\Xcal\setminus\Ycal$ the induced subgraph
  $G_{\Ycal\cup\{x\}}$ has less connected components than $G_{\Ycal}$.
\end{thm}
\begin{proof}
  First, assume that $\KK$ is algebraically closed.  By~\eqref{eq:decomposition-VG} and
  Lemma~\ref{lem:VGY-is-sirreducible} it suffices to show that the condition on $\Ycal$ stated in the theorem
  characterizes the maximal sets $V_{G,\Ycal}$ in the union~\eqref{eq:decomposition-VG} (with respect to inclusion).
  This follows from Lemma~\ref{lem:VGYcontainsVGZ}.

  If $\KK$ is not algebraically closed, then one can argue as follows: By~\cite{EisenbudSturmfels96:Binomial_Ideals} a
  binomial ideal has a binomial primary decomposition over some extension field
  $\hat\KK=\KK[\alpha_{1},\dots,\alpha_{k}]$.  The algebraic numbers $\alpha_{1},\dots,\alpha_{k}$ are coefficients of
  the defining equations of the primary components.  Let $\CC$ be the algebraic closure of $\KK$.  Since the ideals
  $I_{G,\Ycal}$ are defined by pure differences and since the ideals $\CC\otimes I_{G,\Ycal}$ are the primary
  components of $\CC\otimes I_{G,\Ycal}$ in $\CC\otimes R$ it follows that the ideals $I_{G,\Ycal}$ are already the
  primary components of $I_{G}$ (in other words, the primary decomposition is independent of the base field).
\end{proof}

\begin{rem}[Comparison to~\cite{HHHKR10:Binomial_Edge_Ideals}]
  Theorems~\ref{thm:Gbasis} and~\ref{thm:primary-decomposition} are generalizations of Theorems~2.1 and~3.2
  from~\cite{HHHKR10:Binomial_Edge_Ideals}.  While Theorem~2.1 in~\cite{HHHKR10:Binomial_Edge_Ideals} was proved with a
  case by case analysis, the proof of Theorem~\ref{thm:Gbasis} is much more conceptual.  The proof of
  Theorem~\ref{thm:primary-decomposition} relied on the irreducible decomposition of the corresponding variety.  On the
  other hand, the proof of Theorem~3.2 in~\cite{HHHKR10:Binomial_Edge_Ideals} directly proves the equality of the two
  ideals.
\end{rem}

\subsection*{Acknowledgement}

This work has been supported by the Volkswagen
Foundation and the Santa Fe Institute. Nihat Ay thanks David Krakauer and
Jessica Flack for many stimulating discussions on robustness.

\bibliographystyle{IEEEtranSpers}
\bibliography{general}

\begin{thebibliography}{10}
\providecommand{\url}[1]{#1}
\csname url@samestyle\endcsname
\providecommand{\newblock}{\relax}
\providecommand{\bibinfo}[2]{#2}
\providecommand{\BIBentrySTDinterwordspacing}{\spaceskip=0pt\relax}
\providecommand{\BIBentryALTinterwordstretchfactor}{4}
\providecommand{\BIBentryALTinterwordspacing}{\spaceskip=\fontdimen2\font plus
\BIBentryALTinterwordstretchfactor\fontdimen3\font minus
  \fontdimen4\font\relax}
\providecommand{\BIBforeignlanguage}[2]{{%
\expandafter\ifx\csname l@#1\endcsname\relax
\typeout{** WARNING: IEEEtranS.bst: No hyphenation pattern has been}%
\typeout{** loaded for the language `#1'. Using the pattern for}%
\typeout{** the default language instead.}%
\else
\language=\csname l@#1\endcsname
\fi
#2}}
\providecommand{\BIBdecl}{\relax}
\BIBdecl

\bibitem{AyKrakauer2007:Geometric_Robustness_Theory}
{\sc Ay, N.,  and Krakauer, D.~C.}, ``Geometric robustness theory and
  biological networks,'' \emph{Theory in Biosciences}, vol. 125, no.~2, pp. 93
  -- 121, 2007.

\bibitem{CoxLittleOShea08:Ideals_Varieties_Algorithms}
{\sc Cox, D.~A., Little, J.,  and O'Shea, D.}, \emph{Ideals, Varieties, and
  Algorithms: An Introduction to Computational Algebraic Geometry and
  Commutative Algebra}, 3rd~ed.\hskip 1em plus 0.5em minus 0.4em\relax
  Springer, 2008.

\bibitem{DrtonSturmfelsSullivant09:Algebraic_Statistics}
{\sc Drton, M., Sturmfels, B.,  and Sullivant, S.}, \emph{Lectures on Algebraic
  Statistics}, 1st~ed., ser. Oberwolfach Seminars.\hskip 1em plus 0.5em minus
  0.4em\relax Birkhäuser, Basel, 2009, vol.~39.

\bibitem{EisenbudSturmfels96:Binomial_Ideals}
{\sc Eisenbud, D.,  and Sturmfels, B.}, ``Binomial ideals,'' \emph{Duke
  Mathematical Journal}, vol.~84, no.~1, pp. 1--45, 1996.

\bibitem{Fink11:Binomial_ideal_of_intersection_axiom}
{\sc Fink, A.}, ``The binomial ideal of the intersection axiom for conditional
  probabilities,'' \emph{Journal of Algebraic Combinatorics}, vol.~33, no.~3,
  pp. 455--463, 2011.

\bibitem{HHHKR10:Binomial_Edge_Ideals}
{\sc Herzog, J., Hibi, T., Hreinsd\'{o}ttir, F., Kahle, T.,  and Rauh, J.},
  ``Binomial edge ideals and conditional independence statements,''
  \emph{Advances in Applied Mathematics}, vol.~45, no.~3, pp. 317 -- 333, 2010.

\bibitem{JarrahRaposaLaubenbacher07:Nested_Canalyzing_and_other_Functions}
{\sc Jarrah, A.~S., Raposa, B.,  and Laubenbacher, R.}, ``Nested canalyzing,
  unate cascade, and polynomial functions,'' \emph{Physica D}, vol. 233, no.~2,
  pp. 167 -- 174, 2007.

\bibitem{Kauffman93:Origins_of_Order}
{\sc Kauffman, S.~A.}, \emph{The Origins of Order: Self-Organization and
  Selection in Evolution}.\hskip 1em plus 0.5em minus 0.4em\relax Oxford
  University Press, 1993.

\bibitem{Ohtani09:Ideals_of_some_2-minors}
{\sc Ohtani, M.}, ``\BIBforeignlanguage{English}{{Graphs and ideals generated
  by some 2-minors.}}'' \emph{\BIBforeignlanguage{English}{Commun. Algebra}},
  vol.~39, no.~3, pp. 905--917, 2011.

\bibitem{Sturmfels91:GB_of_toric_varieties}
{\sc Sturmfels, B.}, ``Gröbner bases of toric varieties,'' \emph{Tohoku
  Mathematical Journal}, vol.~43, pp. 249 -- 261, 1991.

\bibitem{SwansonTaylor11:Minimial_Primes_of_CI_Ideals}
{\sc Swanson, I.,  and Taylor, A.}, ``Minimal primes of ideals arising from
  conditional independence statements,'' \emph{Preprint: arXiv:1107.5604v3},
  2011.

\end{thebibliography}

\end{document}